\documentclass{article}
\usepackage{amsmath, amssymb, amsthm, graphicx, bm, url,color}

\DeclareMathOperator{\im}{im}

\DeclareMathOperator{\Lag}{Lag}

\theoremstyle{plain}
\newtheorem{thm}{Theorem}[section]

\newtheorem{lem}[thm]{Lemma}

\newtheorem{cor}[thm]{Corollary}
\theoremstyle{definition}

\newtheorem{dfn}[thm]{Definition}

\newtheorem{exam}[thm]{Example}
\newtheorem{rem}[thm]{Remark}
\newtheorem{hp}[thm]{Hypothesis}

\numberwithin{equation}{section}
\numberwithin{figure}{section}

\def\e#1\e{\begin{equation}#1\end{equation}}
\def\iz#1\iz{\begin{itemize}#1\end{itemize}}
\def\ea#1\ea{\begin{align}#1\end{align}}
\def\eq{\eqref}
\def\l{\label}
\def\0{\hspace{0pt}}

\def\Lh{\,\widehat{\!L}}

\def\Uh_#1{\,\widehat{\!U}_{\!#1}}

\def\fm{\mathfrak m}
\def\fp{\mathfrak p}

\def\px{\approx}

\def\T*#1{T^{*\!\!}#1}

\def\id{\mathop{\rm id}\nolimits}

\def\ge{\geqslant}
\def\le{\leqslant\nobreak}

\def\B_#1{\,\ov{\!B_{#1}}}
\def\Ba_#1{\,\ov{\!B_{#1}\!}\,}

\def\cF{{\mathbin{\cal F}}}

\def\={\equiv}

\def\C{{\mathbin{\mathbb C}}}
\def\CP{{\mathbin{\mathbb{C}}}P}

\def\Q{{\mathbin{\mathbb Q}}}
\def\R{{\mathbin{\mathbb R}}}
\def\Z{{\mathbin{\mathbb Z}}}
\def\al{\alpha}
\def\be{\beta}

\def\ep{\epsilon}
\def\la{\lambda}

\def\th{\theta}

\def\om{\omega}
\def\ph{\phi}
\def\Ph{\Phi}
\def\ps{\psi}
\def\Ps{\Psi}
\def\r{\rho}
\def\ta{\tau}

\def\La{\Lambda}

\def\Om{\Omega}

\def\d{{\partial}}
\def\db{\bar\d}
\def\ts{\textstyle}

\def\w{\wedge}
\def\-{\setminus}
\def\bl{\bullet}
\def\op{\oplus}

\def\ot{\otimes}

\def\ov{\overline}

\def\iy{\infty}

\def\t{\times}

\def\hf{{\frac{1}{2}}}

\def\nb{\nabla}
\def\sb{\subseteq}

\def\b{{\mathbf{b}}}

\title{Generalized Thomas--Yau Uniqueness Theorems}
\author{Yohsuke Imagi}
\begin{document}
\maketitle


\begin{abstract}
We generalize Thomas--Yau's uniqueness theorem \cite[Theorem 4.3]{TY} in two ways.
We prove a stronger statement for special Lagranigans and
include minimal Lagrangians in K\"ahler--Einstein manifolds
or more generally $J$-minimal Lagrangians introduced by Lotay and Pacini \cite{LP,LP2}.
In every case the heart of the proof is to make certain Hamiltonian perturbations.
For this we use the method by Imagi, Joyce and Oliveira dos Santos \cite[Theorem 4.7]{IJO}.
\end{abstract}
\section{Introduction}
In this paper we improve and generalize Thomas--Yau's theorem \cite[Theorem 4.3]{TY}.
Our first main result is the following.
For the more complete statement see Corollary \ref{2} (ii).
\begin{thm}\l{01}
Let $X$ be a K\"ahler manifold equipped with a holomorphic volume form.
Let the cohomology Fukaya category $H\cF(X)$ have two isomorphic objects supported near two closed irreducibly-immersed special Lagrangians $L_1,L_2$ respectively.
Then $L_1=L_2\sb X.$
\end{thm}
\begin{rem}
{\bf(i)}
Thomas and Yau prove their uniqueness theorem for closed special Lagrangians $L_1,L_2$ in the same Hamiltonian isotopy class.
But their proof works under the weaker hypothesis as above; that is, we need only that $L_1,L_2$ have the same isomorphism class in $H\cF(X).$
This $H\cF(X)$ is a (usual) associative category obtained from the $A_\iy$ category $\cF(X)$ by taking the ($\fm^1$) cohomology groups of the hom spaces in it.
The key fact is that if $L_1,L_2$ are Hamiltonian isotopic then $L_1,L_2$ will define isomorphic objects in $H\cF(X)$ (after given the additional data including bounding cochains).

It is clear that we need only the weaker hypothesis because Thomas and Yau use in their proof only the Floer cohomology group $HF^*(L_1,L_2),$ the hom space in $H\cF(X).$
But when Thomas and Yan wrote their paper, the theory of Fukaya categories was much less developed at that time, so the statement was presumably better-sounding with Hamiltonian isotopies. 
In this paper we improve Thomas--Yau's theorem by using the more developed theory of Fukaya categories as follows.

{\bf(ii)}
Theorem \ref{01} is stronger than the original theorem in the respect that $L_1,L_2$ need not have (even after shift) the same phase. This would also follow from the following statement: two Lagrangians underlying the same isomorphism class of objects in $H\cF(X)$ should have the same homology class in $X;$ although we shall not discuss this in the present paper. 
The fact that $L_1,L_2$ have the same shift up to shift will be proved in Corollary \ref{7} (before giving the full proof of the theorem), for which we shall use some simple fact about Maslov indices and some nonvanishing results for Floer cohomology groups (which we recall in \S\ref{sec3}). 

Also we do not suppose that either $L_1$ or $L_2$ itself underlies an object of $H\cF(X)$
nor do we suppose even that it is cleanly immersed as Fukaya \cite{Fuk2} does.
But we do suppose that we can perturb $L_1,L_2$ both to generically immersed Lagrangians
which underlie objects of $H\cF(X).$
This is how we deal with badly immersed Lagrangians in the Floer-theory context in this paper.

{\bf(iii)}
The idea of the proof of Theorem \ref{01} is the same as that of Thomas and Yau. But we include the modification by Imagi, Joyce and Oliveira dos Santos \cite[Theorem 4.7]{IJO} because we can unfortunately not justify the original Morse-theory argument \cite[Theorem 4.3]{TY}.
\end{rem}

Another natural generalization is to
minimal Lagrangians in K\"ahler--Einstein manifolds
or more generally to $J$-minimal Lagrangians introduced by Lotay and Pacini \cite{LP,LP2}.
At the moment we state the result only in the K\"ahler--Einstein case.
For the more complete statement see Corollary \ref{2} (i).
\begin{thm}\l{02}
Let $X$ be a $\Z_k$-graded K\"ahler--Einstein manifold of complex dimension $n\le k;$
and $L_1,L_2\sb X$ two closed irreducibly-immersed $\Z_k$-graded minimal Lagrangians with the same grading on $L_1\cap L_2.$
Let $\b_1,\b_2\in H\cF(X)$ be two objects supported respectively near $L_1,L_2$ and such that:
\ea
\l{11}
&\text{either $\b_1\cong\b_2$ or $k\ge 2n-1$ and $\b_1\cong\b_2$ up to shift; and}\\
\l{12}
&HF^i(\b_1,\b_1)\cong HF^i(\b_2,\b_2)\ne0\text{ for } i=0,n.
\ea
Then $L_1=L_2\sb X.$
\end{thm}
\begin{rem}
The more precise meaning of $L_1,L_2$ having the same grading on $L_1\cap L_2$ is given in Definition \ref{dfn gr} below.
We make a more geometric definition of gradings than Seidel \cite{Sd1} does.
\end{rem}

In the circumstances of Theorem \ref{02},
if $c_1(X)>0$ there may be many automorphisms of $X$
and many zero objects of $H\cF(X)$ especially in the toric case as Fukaya, Oh, Ohta and Ono \cite{FOOOt} study.
So we can hardly expect a general uniqueness statement to hold.
But for instance, to the {\it Clifford torus} 
\e\l{CT}T^n:=\{[z_0,\cdots,z_n]\in\CP^n:|z_0|=\cdots=|z_n|\}\e
we can apply Theorem \ref{02} with the latter alternative of \eq{11}.
We can in fact take $k=2n+2$ as we show in Example \ref{ExCliff} below. 
Also according to Fukaya, Oh, Ohta and Ono \cite{FOOOt}
there are $(n+1)$ objects $\b\in H\cF(X)$ supported on $T^n$
with $HF^*(\b,\b)\cong H^*(T^n).$ 
So \eq{12} holds and we have:
\begin{cor}\l{3}
Let $L\sb\CP^n$ be a closed irreducibly-immersed $\Z_{2n+2}$-graded minimal Lagrangian such that $L,T^n$ have the same grading on $L\cap T^n.$
Let $H\cF(X)$ have an object which is supported near $L$ and isomorphic to one of the $(n+1)$ objects $\b$'s above $($all supported on $T^n).$
Then $L=T^n.$
\end{cor}

On the other hand, if $c_1(X)\le0$ the nonzero condition \eq{12} holds automatically in certain circumstances.
For instance, if either $L_1$ or $L_2$ is embedded then \eq{12} holds automatically by
Fukaya, Oh, Ohta and Ono's theorem \cite[Theorem E]{FOOO}.
If $c_1(X)=0$ there is a stronger version of it which we use for Theorem \ref{01}. 

If $c_1(X)<0$ Theorem \ref{12} extends to $J$-minimal Lagrangians.
They are locally unique, so the result is interesting only in the global context.

For the proofs we follow in outline that by Thomas and Yau \cite[Theorem 4.3]{TY}.
In every case the heart of the proof is to
make certain Hamiltonian perturbations of the two Lagrangians.
The Floer-theory condition implies that the two perturbed Lagrangians have at least one intersection point of index $0$ or $n$
at which we do by analysis a sort of unique continuation.

But for the Hamiltonian perturbation process we use the method by
Imagi, Joyce and Oliveira dos Santos \cite[Theorem 4.7]{IJO} as mentioned above.
The modified method works only in the real analytic category.
This causes no problem in the K\"ahler--Einstein case because
in that case everything is automatically analytic by elliptic regularity.
On the other hand, in Theorem \ref{01} and in the $J$-minimal version of Theorem \ref{02}
we work in the $C^\iy$ category
and the real analyticity condition is achieved by another Hamiltonian-perturbation process.

We begin in \S\ref{sec2} with our geometric treatment of graded Lagrangians.
In \S\ref{sec2.5} we recall the relevant facts about Maslov forms in Lotay--Pacini's sense.
In \S\ref{sec3} we give a minimum account of Floer theory which we use in this paper. 
In \S\ref{sec4} we carry out the key step of making Hamiltonian perturbations.
In \S\ref{sec5} we state and prove all our results.
They are essentially at the chain level because \S\ref{sec4} is independent of Floer theory.


Finally we remark that Thomas--Yau's uniqueness theorem is just a bit of their whole proposal \cite{TY} which is now improved by Joyce \cite{Jc}.

{\it Acknowledgements.}
The author would like to thank Mohammed Abouzaid and Kenji Fukaya for helpful conversations.
This work was supported financially by the National Natural Science Foundation of China (11950410501).

\section{Geometric Gradings}\l{sec2}
We make a geometric definition of graded Lagrangians:
\begin{dfn}\l{dfn gr}
Let $(X,\om)$ a symplectic manifold, $J$ an $\om$-compatible almost-complex structure on $X$
and $K_X$ the canonical bundle over $(X,J).$
Let $k$ be either a positive integer or infinity, write $\Z_k:=\Z/k\Z$ for $k$ finite and write $\Z_\iy:=\Z.$
By a {\it $\Z_k$-grading} of $(X,\om,J)$ we mean for $k<\iy$ the pair of a complex line bundle $K_X^{2/k}$ and a bundle isomorphism $(K_X^{2/k})^k\cong K_X^2,$
and for $k=\iy$ a nowhere vanishing section $\Ps$ of $K_X^2.$ 
Suppose given such a $\Z_k$-grading of $X$ and a Hermitian metric $g$ on $(X,J).$
Then for every $x\in X$ and for every Lagrangian plane $\ell\sb T_xX$
there is a $g$-orthogonal decomposition $T_xX=\ell\op J\ell$ and accordingly
a canonical $g$-unit section of $K_X^2$ which we denote by $\Om_\ell^2;$
that is, if $e_1,\cdots,e_n\in\ell$ are a $g$-orthonormal basis then 
\[\Om_\ell^2:=\pm e^1\w\cdots\w e^n\w Je^1\w\cdots\w Je^n.\]
There is also a canonical $k$-fold cover $\Lag^kTX$ of the Lagrangian Grassmannian $\Lag TX$ defined for $k<\iy$ by
\[\Lag^kTX:=\{(\ell,\al)\in\Lag TX\t_X K_X^{2/k}:\Om_\ell^2=\al^k\}\]
and for $k=\iy$ by $\Lag^\iy TX:=\{(\ell,\ph)\in\Lag TX\t_X \R:\Om_\ell^2=e^{i\ph}\Ps\}.$ 
By a $\Z_k$-grading of an immersed Lagrangian $L\sb (X,\om)$
we mean a lift to $\Lag^kTX$ of the tangent-space map $L\to \Lag TX,$
that is, for $k<\iy$  a section $\al:L\to K_X^{2/k}$ with $\al^k=\Om_L^2$
and for $k=\iy$ a section $\ph:L\to S^1$ with $e^{i\ph}\Ps=\Om_L^2.$
It is unique up to $\Z_k$-shifts,
where the shift $[1]$ is defined for $k<\iy$ by $(L,\al)[1]:=(L,e^{2\pi i/k}\al)$
and for $k=\iy$ by $(L,\ph)[1]:=(L,\ph+2\pi).$
\end{dfn}
\begin{rem}
The advantage of this definition is that given two Lagrangians $L_i$ $(i=1,2)$ graded by $\al_i$ or $\ph_i$ as above,
we can compare the two gradings $\al_i$ or $\ph_i.$
In particular, it makes sense to say that they are equal or not.
\end{rem}

We make another definition in the circumstances of Definition \ref{dfn gr}:
\begin{dfn}
For $i=1,2$ let $L_i\sb(X,\om)$ be a Lagrangian with a $\Z_k$-grading $\al_i:L\to K_X^{2/k},$
$k<\iy.$
Suppose $L_1,L_2$ intersect transversely at a point $x\in L_1\cap L_2.$
We define the Maslov index $\mu_{L_1,L_2}(x)\in\Z_k.$
Define $e^{i\ta}\in S^1$ by $\al_2=e^{i\ta}\al_1\in K_X^{2/k}$ over $x.$
Take an isomorphism $(T_xX,J|_x,g|_x)\cong\C^n$
which maps $T_xL_1$ to $\R^n\sb\C^n$ and $T_xL_2$ to
$\{(e^{i\th_1}x_1,\cdots,e^{i\th_n}x_n)\in\C^n:x_1,\cdots,x_n\in\R\}.$
for some $\th_1,\cdots,\th_n\in(0,\pi).$
These $\th_1,\cdots,\th_n$ are unique up to order
and we have
$\Om_{L_2}^2=\pm e^{-2i(\th_1+\cdots+\th_n)}\Om_{L_1}^2.$
So we can define
\begin{equation}
\l{mu}\mu_{L_1,L_2}(x):=\frac{1}{2\pi}(2\th_1+\cdots+2\th_n+k\ta) \in\Z \text{ modulo }k\Z.
\end{equation}
If $k=\iy$ and if for $i=1,2$ the $L_i$ is $\Z$-graded by $\ph_i:L_i\to \R$ then
\e\l{mu2}\mu_{L_1,L_2}(x):=\frac{1}{2\pi}(2\th_1+\cdots+2\th_n-\ph_1(x)+\ph_2(x)) \in\Z.\e
\end{dfn}

We recall now the notion of {\it special} Lagrangians: 
\begin{dfn}\l{dfn slag}
Let $(X,\om)$ a symplectic manifold, $J$ an $\om$-compatible almost-complex structure on $X$
and $\Ps$ a $\Z$-grading of $(X,\om,J).$
We say that an immersed Lagrangian in $(X,\om,J,\Ps)$ is {\it special} if it is $\Z$-graded by a constant.
\end{dfn}
\begin{rem}
In the symplectic context including this definition and \S\ref{sec3} we do not need $J$ to be integrable.
But in the more geometric context we suppose $J$ integrable
and Harvey--Lawson's theorem applies as in Theorem \ref{HL}.
\end{rem}
We prove a lemma which we use in Corollary \ref{7}:
\begin{lem}\l{6.9}
Let $(X,\om,J,\Ps)$ be such as in Definition \ref{dfn slag}
and $L_1,L_2\sb X$ two mutually-transverse special Lagrangian submanifolds which have not even after shifts the same grading.
Then there exists $i\in\Z$ such that the Maslov indices of $L_1,L_2$ all belong to $[i,i+n-1].$
This will still hold under Lagrangian perturbations of $L_1,L_2.$
\end{lem}
\begin{proof}
Let $\ph_1,\ph_2\in\R$ be gradings of $L_1,L_2.$
Then $\ph-\ph'\notin\Z$ and this condition is preserved under Lagrangian perturbations of $L_1,L_2.$ 
Let $x\in L_1\cap L_2$ and let $\th_1,\cdots,\th_n\in(0,\pi)$ be as in \eq{mu}.
Then by \eq{mu2} we have $\mu_{L_1,L_2}(x)\in[i,i+n-1]$ for some $i\in\Z.$
This $i$ is independent of $x$ which completes the proof. 
\end{proof}

We prove another lemma which we use in \S\ref{sec5}:
\begin{lem}\l{P0}
Let $(X,\om)$ be a symplectic manifold of real dimension $2n,$ $\Z_k$-graded with $k\ge n,$
equipped with a compatible almost complex-structure $J$ and equipped with a Hermitian metric $g.$
Let $L_1,L_2\sb X$ be two closed immersed $\Z_k$-graded Lagrangians with the same grading on $L_1\cap L_2.$
Denote by $S\sb L_1\cap L_2$ the set of points at which $L_1,L_2$ have at least one common tangent space.
Then for every open neighbourhood $U\sb X$ of $S$
there exists $\ep>0$ such that every Hamiltonian $\ep$-perturbation of $L_2$ that intersects $L_1$ generically $($that is, only at transverse double points$)$
has no intersection point with $L_1\cap U$ of index $0$ or $n$ modulo $k\Z.$
\end{lem}
\begin{proof}
If this fails there are an open neighbourhood $U\sb X$ of $S,$
a sequence of Hamiltonian perturbations $L_2^i$ of $L_2$ all transverse to $L_1$ and tending smoothly to $L_2,$
and a sequence of intersection points $x^i\in L_1\cap L_2^i\-U$ of index $0$ or $n.$
By hypothesis $L_1\cap L_2^i\-U$ is compact, so there is a subsequence of $x^i$ tending to some point $x\in L_1\cap L_2\-U.$
Suppose now that $k<\iy;$
the other case $k=\iy$ may be treated in the same manner.
Give $L_2^i$ a grading $\al^i$ by the obvious homotopy,
and write $\al^i=\exp(\sqrt{-1}\ta^i)\al$ over $x^i.$
By hypothesis $L_1,L_2$ have the same grading at $x^i$ so we may suppose that for each $i$ we have $|\ta^i|<2\pi/k$ and the sequence $\ta^i$ tends to $0$ as $i\to\iy.$
Denote by $\th_1^i,\cdots,\th_n^i\in(0,\pi)$ the $n$ angles between the two tangent spaces at $x^i.$
Then
$2\th_1^i+\cdots+2\th_n^i+k\ta^i=0\text{ or }n\pi\text{ modulo }2k\pi\Z.$
In fact, since $k\ge n$ it follows that this holds without modulo $2k\pi\Z.$
So the limits of the $n$ angles sum up to $0$ or $n\pi;$
that is, at the limit point $x$ there is a common tangent space to $L_1,L_2.$
But this contradicts the definition of $S.$
\end{proof}

\section{Maslov Forms}\l{sec2.5}
Following Lotay and Pacini \cite{LP2} we make:
\begin{dfn}\l{Mdfn}
Let $(X,\om,J)$ be a K\"ahler manifold of complex dimension $n$
and $K_X$ the canonical bundle over $(X,J).$
Let $L\sb (X,\om)$ be a Lagrangian submanifold, and $g$ a Hermitian metric on $(X,J)$
which need not be the K\"ahler metric of $(X,\om,J)$ nor even K\"ahler.
As in Definition \ref{dfn gr} take the canonical $g$-unit section $\Om_L^2$ of $K_X^2$ over $L.$
Denote by $\nb$ the Chern connection on $(X,J,g)$
and by $A$ the real $1$-form on $K_X^2$ over $L$ defined by
$\nb\Om_L^2=2iA\ot\Om_L^2.$
We call $A$ the {\it Maslov form} on $L$ relative to $g$
and say that $L$ is {\it $g$-Maslov-zero} if $A\=0.$
\end{dfn}
Here it is not essential that $J$ is integrable. Many results by Lotay and Pacini \cite{LP,LP2} hold with $J$ an $\om$-compatible almost-complex structure
and with $\nb$ a connection such that $\nb J=\nb g=0.$
But it is convenient for us to suppose $J$ integrable. We can then take the Chern connection; in \S\ref{sec4} we can use the underlying real analytic structure of $(X,J)$
which will be convenient for us to state our results; and there are as in Theorem \ref{HamPert1} nice deformation theory results for special Lagrangians and $J$-minimal Lagrangians which we use in Corollary \ref{3.1}.

As Lotay and Pacini \cite[Theorem 2.4]{LP} prove, if $g$ is K\"ahler the $g$-Maslov-zero condition and the $J$-minimal condition are equivalent.
If also $L$ is $g$-Lagrangian---that is, $g(Jv,v')=0$ for any
$v,v'\in TL$---then $L$ is $J$-minimal if and only if $L$ is $g$-minimal in the ordinary sense.

We work with $g$-Maslov-zero Lagrangians rather than with $J$-minimal Lagrangians because
for special Lagrangians we can take $g$ to be merely {\it conformally} K\"ahler as we recall now.
Let $(X,\om,J)$ a K\"ahler manifold and $\Ps$ a $J$-holomorphic volume form on $X.$
Joyce \cite{Jb} and other authors call $(X,\om,J,\Ps)$ an almost Calabi--Yau manifold but this does not mean that $J$ is merely an almost complex structure;
$J$ is integrable. It means that $\om$ is Andnot necessarily Ricci-flat.
Denote by $g$ the conformally K\"ahler metric on $(X,J)$ associated with $\ps^2\om$ where
$\ps:X\to\R^+$ is defined by
\e\l{ps}\ps^{2n}\om^n/n!=(-1)^{n(n-1)/2}(i/2)^n\Ps\w\ov\Ps.\e
Then for every Lagrangian submanifold $L\sb X$ the canonical $g$-unit section $\Om_L^2$
is of the form $\al\Ps^2|_L$ where $\al:L\to S^1$ is a smooth function.
This $L$ has zero Maslov $1$-form if and only if $\al$ is constant, that is, if and only if $L$ is special.
In these circumstances Harvey--Lawson's theorem \cite{HL} may be stated as follows:
\begin{thm}\l{HL}
If $L$ is special then $L$ is orientable and {\it calibrated} by $\sqrt\al\Ps$ in Harvey--Lawson's sense
where the choice of $\sqrt\al$ corresponds to the orientation of $L.$
Furthermore, $L$ is area-minimizing with respect to $g$ and $J$-minimal in $(X,J,g).$
\end{thm}

Lotay and Pacini \cite[Proposition 4.5]{LP2} prove a formula for the Maslov $1$-form $A$ and the $J$-mean-curvature vector $H_J$:
\e\l{AH}-JA=(H_J+T_J)\lrcorner g\e
where $T_J$ comes from the torsion of the Chern connection $\nb.$
If $L$ is special then by definition we have $A=0$ and $H_J$ is the ordinary mean curvature which is also zero. So by \eq{AH} we have $T_J=0$ too.
Conversely, if $L$ is merely $J$-minimal then $T_J$ may be nonzero.
This justifies us working with $g$-Maslov-zero Lagrangians rather than $J$-minimal Lagrangians.

We compute now the {\it relative first Chern class} $2c_1(X,L)\in H^2(X,L;\Z)$ for $g$-Maslov-zero Lagrangians $L\sb X.$
We recall therefore:
\begin{lem}
Let $(X,\om)$ be a symplectic manifold, $J$ an $\om$-compatible complex structure,
$K_X$ the canonical bundle over $(X,J)$
and $L\sb(X,\om)$ a Lagrangian submanifold.
Then the relative de-Rham class $2c_1(X,L)\in H^2(X,L;\R),$ which is integral, may be represented by the curvature $2$-form of any connection on $K_X^2$
that is flat over $L.$
\end{lem}
So by Definition \ref{Mdfn} we have: 
\begin{cor}\l{Masl}
Let $(X,\om,J)$ be a K\"ahler manifold,
$g$ a Hermitian metric on $(X,J),$
and $L\sb(X,\om)$ a $g$-Maslov-zero Lagrangian submanifold.
Then $2c_1(X,L)$ may be represented by
the curvature $(1,1)$-form of $K_X^2$ relative to the Chern connection on $(X,J,g).$
\end{cor}
\begin{exam}\l{ExCliff}
Let $X=\CP^n,$ $\om$ the Fubini--Study form and $g$ the Fubini--Study metric.
This is K\"ahler--Einstein and the curvature $(1,1)$-form of $K_X$ relative to $g$ is $-(n+1)\om.$
The Clifford torus $T^n\sb \CP^n$ defined by \eq{CT} is a minimal Lagrangian.
It is $J$-minimal and $g$-Maslov-zero so we can apply to it Corollary \ref{Masl}.

We prove that $\CP^n$ and $T^n$ may both be $\Z_{2n+2}$-graded.
Since $H_1(\CP^n,\Z)=0$ it follows according to Seidel \cite[Lemma 2.6]{Sd1} that $\CP^n$ may be $\Z_{2n+2}$-graded because $2c_1(\CP^n)=-2(n+1)[\om].$
Also $T^n$ may be $\Z_k$-graded if and only if $2c_1(\CP^n,T^n)$ is divisible by $k.$
By Corollary \ref{Masl} we have $c_1(\CP^n,T^n)=-(n+1)[\om]\in H^2(\CP^n,T^n;\R).$
Computation shows that $H_2(\CP^n,T^n;\Z)$ is generated by the image of $H_2(\CP^n,\Z)$
and the $n$ discs $D_a\sb\CP^n,$ $a\in\{1,\cdots,n\},$ defined by $|z_a|\le 1$ and $z_b =1$ for every $b\ne a.$
We have
\[\om|_{D_a}=\frac{ni}{2\pi}\frac{dz_a\w d\bar z_a}{(n+|z_a|^2)^2}
\text{ and }\int_{D_a}\om=1.\]
So $[\om]\in H^2(\CP^n,T^n;\R)$ is integral, q.e.d.
\end{exam}

We turn now to the deformation theory for $g$-Maslov-zero Lagrangians.
We recall Lotay--Pacini's lemma \cite[Lemma 4.1]{LP}
which we use in \S\ref{sec4}:
\begin{lem}\l{LP1}
Let $(X,\om,J)$ be a K\"ahler manifold
and $g$ a Hermitian metric on $(X,J).$
Let $L\sb (X,\om,J)$ be a closed immersed $g$-Maslov-zero Lagrangian
and $NL$ the $g$-normal bundle to $L\sb X.$
For every sufficiently small $v\in C^\iy(NL)$ 
denote by $Av$ the $g$-Maslov $1$-form on the graph of $v$
embedded in $X$ by the $g$-exponential map,
and denote by $A'$ the linearization of $A$ at $0\in C^\iy(NL).$
Then for every $v\in C^\iy(NL)$ we have
\e\l{A'}A'v=dd^*\hat v+\sum_{i=1}^ng(T_\nb(v,e_i),e_i)+v\lrcorner F_\nb\e
where $\hat v:=-Jv\lrcorner g,$ $\{e_1,\cdots,e_n\}$ is a $g$-orthonormal frame on $TL,$
$\nb$ the Chern connection of $(X,J,g),$
$T_\nb$ its torsion tensor and $F_\nb$ its curvature $(1,1)$-form for $K_X.$
\end{lem}
\begin{proof}
This is a version of Lotay--Pacini's lemma \cite[Lemma 4.1]{LP}.
They suppose that $(X,J,g)$ is almost K\"ahler
but again this is not essential; the result applies to every almost Hermitian manifold $(X,J,g)$
with connection $\nb$ such that $\nb J=\nb g=0.$
Our formula \eq{A'} is simpler than Lotay--Pacini's in two respects.
One is that their $J$-volume function $\r_J:L\to\R^+$
is identically equal to $1$
and the other is that their projection operator $\pi_L$ is $g$-orthogonal.
These both hold because $L$ is $g$-Lagrangian,
which completes the proof.
\end{proof}

We remark that we can compute the torsion term in the conformally K\"ahler case.
If $g$ is a conformally K\"ahler metric on $(X,J)$ associated with $\ps^2\om$
then $\nb-\d \log\ps^2\ot\id$ is the Levi-Civita connection for $\om.$ 
Hence by computation we see that for any vector fields $u,v$ on $X$ we have
\[T_\nb(u,v)=d\log\ps\llcorner(u\ot v+Ju\ot Jv-v\ot u-Jv\ot Ju).\]
So the torsion term on \eq{A'} is equal to $(n-1)d\log\ps$ and
\e\l{A3}A'v=d(\ps^{n-1}d^*\ps^{1-n}\hat v)+v\lrcorner F_\nb.\e
This has application to deformation theory;
that is, by \eq{A3} we can apply Lotay--Pacini's result \cite[Proposition 4.5]{LP}
with $\ps^{n-1}$ in place of $\r_J.$
So Lotay--Pacini's uniqueness and persistence theorem \cite[Theorem 5.2]{LP} extends to
$g$-Maslov-zero Lagrangians with $g$ conformally K\"ahler.
But the corresponding perturbations are just Lagrangian and not necessarily Hamiltonian as we want
in the Floer theory context.

On the other hand, we can make (domain) Hamiltonian perturbations of special Lagrangians
and $J$-minimal Lagangians.
The result for special Lagrangians goes back to McLean's theorem \cite{ML}
and that for $J$-minimal Lagrangians is proved by Lotay--Pacini \cite[Theorem 5.6]{LP}.
\begin{thm}\l{HamPert1}
{\bf(i)}
Let $(X,\om,J)$ be a K\"ahler manifold, $\Ps$ a $J$-holomorphic volume form on $X$
and $L\sb X$ a closed immersed special Lagrangian relative to $(J,\Ps).$ 
Then for every sufficiently small perturbation $(J',\Ps')$ of $(J,\Ps)$ as $\om$-compatible complex structures and holomorphic volume forms relative to them,
there exists a domain Hamiltonian perturbation of $L$ which is special with respect to $(J',\Ps').$
\\{\bf(ii)}
Let $(X,\om,J)$ be a K\"ahler manifold such that $-\om$ is the Ricci $(1,1)$-form of another K\"ahler metric $g$ on $(X,J),$
and  $L\sb X$ a closed immersed $J$-minimal Lagrangian. 
Then for every sufficiently small perturbation $(J',g')$ of $(J,g)$ as $\om$-compatible complex structures and K\"ahler metrics relative to them
whose Ricci $(1,1)$-forms are all equal to $-\om,$
there exists a domain Hamiltonian perturbation of $L$ which is $J$-minimal with respect to $(J',g').$
\end{thm}

\section{$HF^*$ Nonvanishing Theorems}\l{sec3}
Following Fukaya, Oh, Ohta and Ono \cite{FOOO} given a unital commutative ring $R$ we define the Novikov ring
\e\l{La}\La=\left\{\sum_{i=0}^\iy a_iT^{\la_i}:a_i\in R[e,e^{-1}],\la_i\in\R\text{ for each }i\text{ and }\lim_{i\to\iy}\la_i=+\iy\right\}\e
where $T,e$ are two formal variables. They are related respectively to the areas and Maslov numbers of pseudoholomorphic curves.
We denote by $\La^0\sb\La$ the subring defined by the same formula \eq{La} as $\La$
but with $\la_i\ge0$ in place of $\la_i\in\R.$
We denote by $\La^+\sb\La^0$ the ideal defined by \eq{La} with $\la_i>0$ in place of $\la_i\in\R.$
We write $\px$ for $=$ modulo $\La^+.$

Fukaya categories are defined over $\La.$
If the ambient symplectic manifold is spherically nonnegative we can take $R=\Z$ as Fukaya, Oh, Ohta and Ono \cite{FOOO-Z} do.
But in general we have to do the pseudoholomorphic-curve counts over $\Q$ and we shall therefore suppose that $R$ contains $\Q.$
Otherwise, in our applications we need no condition on $R;$ that is, we work under the following:
\begin{hp}\l{0}
Let $(X,\om)$ be a $\Z_k$-graded symplectic manifold of real dimension $2n,$
either compact or $\om$-convex at infinity.
Let $R$ be a unital commutative ring such that if $X$ is not spherically nonnegative then $\Q\sb R.$
If $2\ne0$ in $R,$ we shall suppose that $X$ is given a background bundle for relative spin structures on Lagrangians in $X,$
and that the Lagrangians concerned are all given relative spin structures.
\end{hp}
Under this hypothesis there is a Fukaya category $\cF(X)$ and its cohomology category $H\cF(X).$
In this paper, following Akaho and Joyce \cite{AJ} we suppose that every object of $\cF(X)$ consists of
a closed generically-immersed Lagrangian $L\sb X$ and its additional structures such as bounding cochains.
More precisely, we should write the immersion map to $L,$ say $\Lh\to L,$ and `closed' means $\Lh$ being a  closed manifold.
We can include also a suitable class of local systems on $L$ and the results in this paper will hold for them but we shall omit this in our formal treatment.
We recall now the notion of bounding cochains now.

We recall first the homologically perturbed version of Floer complexes.
Given a closed manifold $\Lh$ and a generic Lagrangian immersion $\Lh\to L\sb X$ we write $L\cap L:=\Lh\t_X\Lh.$
This is the disjoint union of the diagonal $\Lh$ and self-intersection pairs.
It depends upon the choice of $\Lh\to X$ but is unique up to automorphisms of $\Lh.$
We write
$C_L^*:=H^*(L\cap L,R)[-\mu_{L,L}]$
where the diagonal has no degree shift and every self-intersection pair has degree equal to its Maslov index.
If $L$ is embedded, this $C_L^*$ is $\Z$-graded and supported in degrees $0,\cdots,n.$
If $L$ is $\Z_k$-graded, so is $C_L^*.$
If also $L$ has on its every self-intersection pair two gradings nearly equal,
then $C_L^*$ is supported in degrees $0,\cdots,n$ modulo $k\Z.$

We use the following version of $A_\iy$ algebra structure constructions.
There are on $C_L^0$ a unit and the intersection pairing so we can speak of unitality and cyclic symmetry as
Fukaya \cite{Fuk1} does.
\begin{thm}\l{A}
There is on $C_L^*\ot\La^0$ a cohomologically-unital cyclic
curved $A_\iy$ algebra structure $(\fm^i)_{i=0}^\iy$ with $\fm^1\px0$ and $\fm^2\px\pm\w$ where $\w$ is the cup-product map.
If also $L$ is embedded, we can make this strictly unital. 
\end{thm}
\begin{rem}
The embedded case is proved by Fukaya \cite{Fuk1}.
This will extend to the immersed case if we give up having the strict unit and satisfy ourselves with
a cohomological unit.
To have the strict unit we ought to perturb in a certain manner the relevant pseudoholomorphic-curve moduli spaces.
But in the immersed case there would be moduli spaces of constant maps to
self-intersection points. If we perturbed these, the cyclic symmetry would fail.

The cyclic symmetry is used in Lemma \ref{v} below.
But also we give it another proof in which we use the open-closed map.
\end{rem}

By a {\it bounding cochain} on $(C_L^*\ot\La^0;\fm^0,\fm^1,\fm^2,\cdots)$
we mean an element $\b\in C^1_L\ot\La^+$ with $\sum_{i=0}^\iy\fm^i(\b,\cdots,\b)=0.$
Given such a $\b$ there is a natural way of giving $C_L^*\ot\La^0$
a cohomologically-unital cyclic {\it ordinary} $A_\iy$ algebra structure
$(\fm^i_\b )_{i=0}^\iy$ with $\fm_\b^1\px0$ and $\fm_\b^2\px\pm\w.$

We denote by $1\in C_L^0$ the unit, which is the cohomological unit in Theorem \ref{A}.
We denote by $*1\in C_L^n$ the volume form supported on the diagonal,
which is if $2\ne0$ in $R$ the Poincar\'e dual to a point.
We denote also symbolically by
$(x,y):=\pm\int_{L\cap L}x\w y$
the intersection pairing on $C_L^*.$ The sign does not matter to us. 
We prove now:
\begin{lem}\l{v}
For every bounding cochain $\b\in C_L^1\ot\La^+$ we have:
\iz
\item[\bf(i)] if $L$ is embedded, we have $*1\notin\im\fm^1_\b;$ and
\item[\bf(ii)] if $\fm^1_\b *1=0$ then $[*1]\ne0$ in $H^n(C_L^*\ot\La^0,\fm_\b^1).$ 
\iz
\end{lem}
\begin{proof}
Suppose contrary to (i) that $*1=\fm^1_\b x$ for some $x\in C^*_L\ot\La.$
Then using the strict unit we find
\e\l{v0}
0=(\fm^1_\b 1,x)=\pm(1,\fm^1_\b x)=\pm(1,*1)\px\pm\int_{L\cap L}\fm^2_\b(1,*1)\px\pm\int_{L\cap L}*1=1.
\e
This is impossible which proves (i).

We give (ii) two proofs.
Firstly, if $*1=\fm^1_\b x$ for some $x\in C^*_L\ot\La$ then as in \eq{v0} we have
\e\l{v1}
0\px\pm\int_{L\cap L}\fm_\b^2(1,*1)
\e
but $\fm_\b^2(1,*1)$ need not be $*1.$
However $1$ is a cohomological unit and $*1$ an $\fm_\b^1$-cocycle, so
\e\l{v2}
\fm_\b^2(1,*1)=*1+\fm_\b^1y\px*1\px *1.
\e
Now by \eq{v1} and \eq{v2} we have
$0\px\pm\int_{L\cap L}*1\px\pm1\in R\-\{0\}.$
This is impossible and so completes the first proof.

In the other proof of (ii) we use neither the cyclic symmetry nor the unitality but the open-closed map constructed by Fukaya, Oh, Ohta and Ono \cite[\S3.8]{FOOO}.
This extends to generically immersed Lagrangians.
We use in particular the map
$\fp_\b:H^n(C_L^*\ot\La^0,\fm_\b^1)\to H^{2n}(X,R)\ot\La^0.$
Denote by $*_X1$ the volume form on $X$
and to distinguish it from that on $L$ write $*_L1:=*1\in C^n_L.$
Then by the extended version of Fukaya, Oh, Ohta and Ono's lemma \cite[Lemma 6.4.2]{FOOO}
we have $\fp_\b[*_L1]\px[*_X1].$
So $[*_L1]\ne0.$
\end{proof}

We prove now our version of Fukaya, Oh, Ohta and Ono's theorem \cite[Theorem E]{FOOO}.  
We recall therefore that the $A_\iy$ structure of Theorem \ref{A} is gapped;
that is, if we take on $X$ an $\om$-compatible  almost-complex structure $J$
we can write $\fm_\b^1$ as the sum of $\fm_{\b,\be}^1$ of degree $1-[\be]\cdot2c_1(X,L)$
where $\be\in H_2(X,L;\Z)$ is a genus-zero stable $J$-holomorphic curve class.
Also, given a bounding cochain $\b\in C_L^1\ot \La^+$ we define
$HF^*(\b,\b):=H^*(C_L^*\ot \La,\fm_\b^1)\ot\La.$

\begin{thm}\l{ne0}
Let Hypothesis \ref{0} hold and let $L\sb X$ be a closed generically-immersed Lagrangian.
Suppose also that:
\iz
\item[\bf(i)] $L$ is embedded and there is on $X$ an $\om$-compatible almost-complex structure $J$
such that for every $J$-holomorphic curve class $\be\in H_2(X,L;\Z)$
we have $\be\cdot2c_1(X,L)\le0;$ or
\item[\bf(ii)] $k\ge n+2$ and $L$ has on its every self-intersection pair the two gradings nearly equal.
\iz
Then for $i=0,n$ and for every object $\b\in H\cF(X)$ supported on $L$
we have $HF^i(\b,\b)\ne0.$
\end{thm}
\begin{proof}
The case (i) is due to Fukaya, Oh, Ohta and Ono \cite[Theorem E]{FOOO} but we recall the proof.
Since $L$ is embedded it follows that $C_L^*$ is $\Z$-graded
and supported in degrees $0,\cdots,n.$
Also, since $\be\cdot2c_1(X,L)\le0$ it follows that $\fm_\b^1$ has degree $\ge1.$
So under $\fm_\b^1$ nothing goes to $1$ and $*1$ goes to zero.
Since the cohomological unit is by definition an $\fm_\b^1$-cycle it follows now that $[1]\ne0$ in $H^0(C_L^*\ot\La^0,\fm_\b^1)$ and accordingly in $HF^0(\b,\b).$
On the other hand, we have proved that $*1$ is also an $\fm_\b^1$-cycle.
So by Lemma \ref{v} (i) we have $[*1]\ne0$ in 
$H^n(C_L^*\ot\La^0,\fm_\b^1)$ and accordingly in $HF^n(\b,\b).$

In the case (ii) we have $C_L^*$ merely $\Z_k$-graded.
But $k\ge n+2$ so $C_L^{-1}=C_L^{n+1}=0.$
Now $d_\b$ has degree $1$ modulo $k\Z$ so again
under $\fm_\b^1$ nothing goes to $1$
and $*1$ goes to zero.
So $[1]\ne0$ in $HF^0(\b,\b)$ and now by Lemma \ref{v} (ii) we have
$[*1]\ne0$ in $HF^n(\b,\b).$
This completes the proof.
\end{proof}

We apply Theorem \ref{ne0} to $g$-Maslov-zero Lagrangians.
We say that a closed immersed Lagrangian  $L\sb X$ {\it nearly} underlies $\cF(X)$ if there is an arbitrarily-small Hamiltonian deformation of $L$ which
underlies an object of $H\cF(X).$
Such an object is said to be {\it supported near $L$.} 
\begin{cor}\l{7.1}
Let Hypothesis \ref{0} hold with $X$ given an $\om$-compatible almost-complex structure $J.$
Let $g$ be another Hermitian metric on $(X,J)$ and
$L\sb (X,\om,J,g)$ a close $\Z_k$-graded immersed $g$-Maslov-zero Lagrangian.
Suppose that either:
{\bf(i)} $L$ is embedded; or
{\bf(ii)} $k\ge n+2$ and $L$ has on its every self-intersection pair the two gradings nearly equal.
Then for $i=0,n$ and for every object $\b\in H\cF(X)$ supported near $L$
we have $HF^i(\b,\b)\ne0.$ 
\end{cor}
\begin{proof}
In the case (i) it follows from Corollary \ref{Masl} that the condition (i) of Theorem \ref{ne0} holds.
This concerns only integral homology and cohomology classes and so is preserved under perturbations of $L.$
Taking one of them which underlies $\b$ we see from Theorem \ref{ne0} that $HF^i(\b,\b)\ne0$ as we want. 
Also in the case (ii) the condition (ii) of Theorem \ref{ne0} is preserved under perturbations of $L$
so the conclusion follows in the same way.
\end{proof}

For special Lagrangians we can say more:
\begin{cor}\l{7}
Let Hypothesis \ref{0} hold with $X$ given an $\om$-compatible almost-complex structure $J$ and a $J$-holomorphic volume form.
Then we have:
\iz
\item[\bf(i)]
For $i=0,n$ and for every object $\b\in H\cF(X)$ supported near a closed immersed special Lagrangian,
we have $HF^i(\b,\b)\ne0.$
\item[\bf(ii)]
Let $L_1,L_2\sb X$ be two closed immersed special Lagrangians
and let $H\cF(X)$ have two isomorphic objects $\b_1,\b_2$ supported respectively near $L_1,L_2.$
Then $L_1,L_2$ have up to shift the same grading.
\iz
\end{cor}
\begin{proof}
The part (i) follows immediately from Corollary \ref{7.1} with $k=\iy.$ 
If (ii) fails then by Lemma \ref{6.9} we have $HF^*(\b_1,\b_2)$ supported in degrees $i,\cdots,i+n-1$ for some $i\in\Z.$
So either $HF^0(\b_1,\b_2)=0$ or $HF^n(\b_1,\b_2)=0.$
But $\b_1\cong\b_2$ so either $HF^0(\b_1,\b_1)=0$ or $HF^n(\b_1,\b_1)=0.$
This however contradicts (i).
\end{proof}

\section{Hamiltonian Perturbation Theorem}\l{sec4}
In this section we prove the following theorem.
We work in the real analytic category
and for our applications in \S\ref{sec5} we can take the underlying real analytic structure of a complex manifold $(X,J).$
In fact any other real analytic structure will do if it makes everything analytic
but the statement will then be too long.
\begin{thm}\l{Ham}
{\bf(i)}
Let $(X,\om,J)$ be a real analytic K\"ahler manifold of complex dimension $n$
and $g$ a real analytic Hermitian metric on $(X,J).$
Let $L_1,L_2\sb X$ be two distinct closed irreducibly-immersed $g$-Maslov-zero Lagrangians
and let $S$ be as in Lemma \ref{P0}.
Then there exists a neighbourhood $U$ of $S\sb X$
and arbitrarily-small Hamiltonian deformations $L_1',L_2'$ of $L_1,L_2$ respectively
which intersect generically each other with no intersection point in $U$ of index $0$ or $n.$
This will still hold if $L_1,L_2$ are $g$-Maslov-zero only near $L_1\cap L_2.$
If also Hypothesis \ref{0} holds and $L_1,L_2$ nearly underlie $H\cF(X)$ we can make $L_1',L_2'$ underlie $H\cF(X).$
\\{\bf(ii)}
The assertion will hold even if $\om,g$ are not analytic but if there are
$\om$-compatible complex structures $J_t,$
$J_t$-Hermitian metrics $g_t$ and $g_t$-Maslov-zero Hamiltonian perturbations $L_{1t},L_{2t}$ of $L_1,L_2$
all parametrized smoothly by $t\in\R$ with $(J_0,g_0,L_{10},L_{20})=(J,g,L_1,L_2)$
and such that for $t\ne0$ the $\om,g_t$ are $J_t$-analytic.
\end{thm}
\begin{rem}
{\bf(i)}
The theorem is of local nature so in fact we need the $g$-Maslov-zero condition and the real analyticity condition only near $L_1\cap L_2.$
\\{\bf(ii)}
The index in $U$ may be defined without grading $L_1',L_2';$ that is, near $S$ we write $L_2'$
over $L_1'$ as the graph of an exact $1$-form $df$ and we take the Morse indices of $f.$
This definition agrees with that in \S\ref{sec2}
if $L_1,L_2$ are graded with the same grading on $L_1\cap L_2$ and if $L_1',L_2'$ are graded by the obvious homotopies.
%
\end{rem}


There are a few key estimates which we use for the proof of Theorem \ref{Ham}.
Firstly, according to \L{}ojasiewicz \cite{Lj}
for every function $f:\R^n\to\R$ with $f(0)=0$ and {\it analytic} near $0\in\R^n$
there exist two constants $c>0$ and $p\in(0,1)$ such that near $0\in\R^n$
we have $|f|^p\le c|df|.$
This implies readily that we have:
\begin{lem}\l{Loj}
In the circumstances of Theorem \ref{Ham} {\rm(i)}
take a Weinstein neighbourhood of $L_2\sb(X,\om)$ and write $L_1$ near $S$ as the graph over $L_2$ of a closed $1$-form $u.$
Then near $S$ we have
$|u|\ll |\nb u|\ll1;$
that is, for every $\ep>0$ there exists a neighbourhood of $S\sb L_2$ on which $|u|\le \ep|\nb u|\le\ep^2.$
Here $|\bl|,\nb$ are both computed with respect to the induced metric on $L_2.$
\end{lem}

We show next that the Maslov $1$-form operator may be well approximated by $dd^*$:
\begin{lem}\l{Hlem}
Let $(X,\om,J)$ be a K\"ahler manifold and let $g$ be a Hermitian metric on $(X,J).$
Let $L\sb X$ be a closed immersed $g$-Maslov-zero Lagrangian
and let $NL$ be the $g$-normal bundle to $L\sb X.$
Then there exists a constant $c>0$ such that for every sufficiently $C^1$-small $v\in C^\iy(NL)$
whose graph $($embedded in $X$ by the $g$-exponential map$)$ is $g$-Maslov-zero, we have at every point of $L$
\e\l{H0}|dd^*\hat v|\le c|v|_{C^1}(1+|v|_{C^2})\e
where $\hat v$ is the same as in Lemma \ref{LP1},
$|v|_{C^1}:=|v|+|\nb v|$
and $|v|_{C^2}:=|v|+|\nb v|+|\nb^2v|;$
the latter two are computed pointwise on $L$ with respect to $g.$
\end{lem}
\begin{proof}
Let $v\in C^\iy(NL)$ be sufficiently $C^1$-small
and $Av\in C^\iy(T^*L)$ the Maslov $1$-form of the graph of $v.$
The tangent space to the graph of $v$ involves the first derivative $\nb v$
so the canonical section of $K_X^2$ over the graph of $v$ is of the form $B(v,\nb v)$
where $B:NL\op(T^*L\ot NL)\to NL$ is a smooth function. 
Then $Av$ involves the differentiation of $B(v,\nb v)$ so 
\e\l{H0.5}Av=A_1(v,\nb v)+A_2(v,\nb v)\cdot\nb^2v\e
where $A_1:NL\op(T^*L\ot NL)\to NL$ and $A_2:NL\op(T^*L\ot NL)\to NL\ot TL\ot TL$ are some smooth functions.
Denote by $A'$ the linearization of $A$ at $0\in C^\iy(NL).$
Then by \eq{H0.5} there is a $v$-independent constant $c_1>0$ such that
\e\l{H1}|(A-A')v|\le c_1[|v|^2+|\nb v|^2+|\nb^2 v|(|v|+|\nb v|)]\le c_1|v|_{C^1}|v|_{C^2}\e
where the same constant $c_1$ will do for both the inequalities.
Now by Lemma \ref{LP1} $A'$ is, up to lower-order terms, equal to $dd^*\hat v;$
that is, there is a $v$-independent constant $c_2>0$ such that
\[|A'v-dd^* \hat v|\le  c_2|v|_{C^1}.\]
So by \eq{H1} we have $|(A-A')v|\le \max\{c_1,c_2\}|v|_{C^1}(1+|v|_{C^2}).$
This implies indeed that if $Av=0$ then \eq{H0} holds.
\end{proof}
We give a corollary to the lemma above:
\begin{cor}\l{Hcor}
In the circumstances of Lemma \ref{Hlem}
take a Weinstein neighbourhood of $L\sb X.$
Then there exists a constant $c>0$ such that for every sufficiently $C^1$-small $1$-form $u\in C^\iy(T^*L)$ whose graph is also minimal, we have at every point of $L$
\e\l{H5}|dd^* u|\le c|u|_{C^1}(1+|u|_{C^2}).\e
\end{cor}
\begin{proof}
The two embeddings, one given by the exponential map and the other given by the Weinstein neighbourhood, are mutually related by a diffeomorphism $F$ between neighbourhoods of the zero sections of $T^*L,NL$ respectively which induces the identity on the zero section $L,$ and over every point of $L,$ a diffeomorphism of the two fibres.
So there is a smooth function $F_1:T^*L\to NL\ot TL$
such that for every sufficiently $C^1$-small $u\in C^\iy(T^*L)$ we have
$Fu=F_1u\cdot u.$
In particular there is a $u$-independent constant $c>0$ such that
\e\l{H6} |Fu|_{C^1}\le |u|_{C^1}\text{ and }|Fu|_{C^2}\le c|u|_{C^2}.\e
The function $Gu:=-JFu\lrcorner g$ is also of the same form; that is,
there is a smooth function $G_1:T^*L\to T^*L\ot TL$ such that
$Gu=G_1u\cdot u.$
This implies in turn that we have, making $c$ large enough,
\e\l{H7}
|dd^*Gu|\le c|dd^*u|+c|u|_{C^1}(1+|u|_{C^2})
\e
Now by \eq{H0} we have, making $c$ large enough,
\[ |dd^* Gu|\le c|Fu|_{C^1}(1+|Fu|_{C^2}).\]
So by \eq{H7} we have \eq{H5} with $c$ large enough.
\end{proof}

Now we prove Theorem \ref{Ham} in three steps:
\begin{proof}[Step 1: proof of the first part of {\rm(i)}]
In this case we take $L_1'=L_1$ and perturb only $L_2.$
We take a Weinstein neighbourhood of $L_2\sb (X,\om)$
and define $L_2'$ as the graph of a certain exact $1$-form $df$ on $L_2$ with $f:L_2\to\R$ taken as follows.
Since $L_1,L_2$ are both connected and analytic with $L_1\ne L_2$ as subsets of $X$ it follows that $L_2\-L_1$ is dense in $L_2.$
So there is a Morse function $f:L_2\to\R$ with $df\ne0$ on $L_1\cap L_2.$
Since $S\sb L_1\cap L_2$ it follows then that
there is a constant $c>0$ such that we have near $S$
\e\l{f}|f|_{C^2}<c|df|.\e
This condition is preserved under rescalings and smooth perturbations of $f$
so we can make $L_2'$ arbitrarily $C^\iy$ close to $L_2$ and transverse to $L_1.$

Near $S$ take a local component of $L_1$ if need be (in the strictly immersed case) and write it as the graph over $L_2$ of some closed $1$-form $u.$
This is possible by the definition of $S$ and in fact we have also $u=\nb u=0$ on $S.$
So by Lemma \ref{Loj} we have near $S$
\e\l{ph1}|u|\ll|\nb u|\ll1.\e
On the other hand, by Corollary \ref{Hcor} with $L_2$ in place of $L,$ there is a constant $c'>0,$ which depends upon $|u|_{C^2},$
such that we have near $S$
\e\l{ph2} |dd^*u|\le c' (|u|+|\nb u|)\le 2c'|\nb u|\ll1\e
where the second inequality follows from \eq{ph1}.
Applying to $d^*u$ the \L{}ojasiewicz estimate, we find a constant $\ep>0$ such that we have near $S$
\e\l{ph3} |d^*u|\le |dd^*u|^{1+\ep}\le (2c'|\nb u|)^{1+\ep}\ll|\nb u|\e
where the second inequality follows from \eq{ph2}.

We show now that we have on $L_1\cap L_2'$ near $S$
\e
\l{Hs}|d^*(u-df)|\le n^{-1/2}|\nb(u-df)|.
\e
On $L_1\cap L_2'$ near $S$ we have $u=df$ and so by \eq{f} 
\e\l{phf}|\nb(u-df)|\ge|\nb u|-|\nb df|\ge|\nb u|-c|df|=|\nb u|-c|u|\ge\frac{1}{2}|\nb u|\e
where the last inequality follows from \eq{ph1}.
On the other hand by \eq{ph3} we have on $L_1\cap L_2'$ near $S$
\e\l{phf2}
|d^*(u-df)|\le|d^*u|+c|df|
\le\frac{1}{4}n^{-1/2}|\nb u|+c|u|
\le\frac{1}{2}n^{-1/2}|\nb u|
\e
where the last inequality follows again from \eq{ph1}.
By \eq{phf} and \eq{phf2} we have \eq{Hs} as we want.

Finally we take a point $x\in L_1\cap L_2'$ near $S$
and compute its index,
which is the index of the nondegenerate symmetric bilinear form $\nb(u-df)$ on $T_xL_2.$
Diagonalize it by an orthonormal basis of $T_xL_2$ and denote by $\la_1,\cdots,\la_n$ the diagonal entries.
Then by \eq{Hs} we have
\[\ts|\la_1+\cdots+\la_n|\le n^{-1/2}(\la_1^2+\cdots+\la_n^2)^{1/2}\le\max\{|\la_1|,\cdots,|\la_n|\}.\]
But $\la_1,\cdots,\la_n$ are all nonzero so they have not all the same sign; that is, the index is neither $0$ nor $n$ as we want.
\end{proof}

\begin{proof}[Step 2: proof of the latter part of {\rm(i)}]
We make first a generic Hamiltonian perturbation of $L_2$ which underlies an object of $H\cF(X)$
and we define $L_2'$ as a further Hamiltonian perturbation of it.
As in step 1 we take a Weinstein neighbourhood of $L_2$ and define $L_2'$ as the graph of some exact $1$-form $dh_2$ over $L_2$
with $h_2\ne0$ on $S.$
Also as in step 1 write $L_1$ near $S$ as the graph of some real analytic closed $1$-form $u$ on $L_2$ near $S.$
Take a generic Hamiltonian perturbation $L_1'$ of $L_1$ which underlies $H\cF(X)$
and is close enough to $L_1$ to be written near $S$ as the graph of $u+dh_1$ over $L_2$ with
\e\l{ph5}|h_1|_{C^2}\ll |h_2|_{C^2}.\e
Since $dh_2\ne 0$ on $S$ it follows that there is a constant $c>0$ such that near $S$
we have $|h_2|_{C^2}<\frac{c}{4}|dh_2|.$
And then putting $f:=h_2-h_1$ we have
\e|h_2|_{C^2}<\frac{c}{4}(|df|+|dh_1|)\le\frac{c}{4}|df|+\frac{1}{2}|h_2|_{C^2}\l{ph6}\e
where the last inequality follows from \eq{ph5}. The whole estimate \eq{ph6} implies
$|h_2|_{C^2}<\frac{c}{2}|df|.$
Hence using again \eq{ph5} we find
\[|f|_{C^2}\le |h_2|_{C^2}+|h_1|_{C^2}<
\frac{c}{2}|df|+|h_2|_{C^2} <c|df|;\]
that is, the estimate \eq{f} holds as in step 1.
Following the subsequent estimates we see also that \eq{Hs} holds
now on $L_1'\cap L_2'$ near $S.$
This implies again that $L_1',L_2'$ have near $S$ no intersection point of index $0$ or $n.$
\end{proof}

\begin{proof}[Step 3: proof of Theorem \ref{Ham} {\rm (ii)}]
As in step 2 we make Hamiltonian perturbations $L_1',L_2'$ of $L_1,L_2$ both underlying $H\cF(X).$
The $L_2'$ is the graph of $dh_2$ over $L_2$ and
the $L_1'$ near $S$ is the graph of $u+dh_1$ over $L_2$ near $S.$
The difference $f:=h_2-h_1$ satisfies \eq{f} as in step 1.
We extend these to smooth families. By Theorem \ref{HamPert1} there are $g_t$-Maslov-zero Hamiltonian perturbations
$L_{1t},L_{2t}$ of $L_1,L_2.$
Extend the Weinstein neighbourhood of $L_2$ to those of $L_{2t}$
and write $L_2'$ as the graph of some $dh_{2t}$ over $L_{2t}.$
Denote by $S_t\sb L_{1t}\cap L_{2t}$ the set of intersection points at which $L_{1t},L_{2t}$ have at least one common tangent space.
Write $L_{1t}$ near $S_t$ as the graph of a real analytic closed $1$-form $u_t$ over $L_{2t}$ near $S_t$
and write $L_1'$ near $S_t$ as the graph of $u_t+dh_{1t}$ over $L_{2t}$ near $S_t.$
Then with $t$ small enough the estimate \eq{f} holds for $f_t:=h_{2t}-h_{1t}$ in place of $f$
and we can follow the subsequent estimates in step 1.
We can do this in a $t$-independent neighbourhood $U$ of $S\sb X$ because $S_t$ tends to $S.$
So $L_1',L_2'$ have in $U$ no intersection point of index $0$ or $n.$
\end{proof}

\section{Conclusions}\l{sec5}
From Lemma \ref{P0} and Theorem \ref{Ham} we deduce:
\begin{thm}\l{P1}
Let $(X,\om,J)$ be a real analytic K\"ahler manifold of complex dimension $n,$ $\Z_k$-graded with $k\ge n$ and given 
another real analytic Hermitian metric $g.$
Let $L_1,L_2\sb X$ two distinct closed irreducibly-immersed $g$-Maslov-zero Lagrangians with the same grading on $L_1\cap L_2.$
Then there exist arbitrarily-small Hamiltonian deformations $L_1',L_2'$ of $L_1,L_2$
which intersect generically each other with no intersection point of index $0$ or $n.$
If also Hypothesis \ref{0} holds and $L_1,L_2$ nearly underlie objects of $H\cF(X)$ then we can make $L_1',L_2'$ underly objects of $H\cF(X).$
\end{thm}
\begin{proof}
Let $U$ be as in Theorem \ref{Ham} and corresponding to this $U$ let $\ep$ be as in Lemma \ref{P0}.
Then corresponding to this $\ep$ let $L_i'$ $(i=1,2)$ be such $\ep$-perturbations of $L_i$ as in Theorem \ref{Ham}.
Applying Lemma \ref{P0} to the intersection points outside $U$
and Theorem \ref{Ham} to those in $U$
we see that $L_1',L_2'$ are such as we want.
\end{proof}

We make now the $C^\iy$ version of Theorem \ref{P1}:
\begin{thm}\l{3.1}
Let $(X,\om,J)$ be a K\"ahler manifold of complex dimension $n$
which is either: $\Z$-graded by a $J$-holomorphic volume form $\Ps$ and given the conformally K\"ahler metric $g$ as in $\S\ref{sec2}$ by using $\eq{ps};$
or $\Z_k$-graded with $k\ge n$ and given another K\"ahler metric $g$
whose the Ricci $(1,1)$-form is $-\om.$ 
Let $L_1,L_2\sb (X,\om,J,g)$ be two distinct closed irreducibly-immersed $g$-Maslov-zero Lagrangians with the same grading on $L_1\cap L_2.$
Suppose that one of the following three conditions holds: {\bf(i)} $X$ is compact;  {\bf(ii)} $X$ is $\om$-convex, $\om$ is exact and $J$ is Stein; or
{\bf(iii)} either $L_1$ or $L_2$ is generically immersed.
Then the same conclusion holds as in Theorem \ref{P1} (both the two sentences in it).
\end{thm}
\begin{proof}
We reduce the problem to the real analytic case by a further perturbation process.
In the case (i) we use the fact that
for every compact K\"ahler manifold $(X,\om,J)$
there exists on $(X,J)$ a smooth family $(\om_t)_{t\in\R}$ of K\"ahler forms with $\om_0=\om$ and $\om_t,$ $t\ne0,$ analytic.
This seems well known but for clarity we give it a proof.
Denote by $g_\om$ the K\"ahler metric of $(X,\om,J)$
and take on $(X,J)$ a smooth family of {\it Riemannian} metrics $g_t'$
with $g_0'=g_\om$ and $g_t',$ $t\ne0,$ analytic.
Then $g_t:=\hf(g_t'+J^*g_t')$ defines on $(X,J)$ a smooth family of Hermitian metrics with
$g_0=g_\om$ and $g_t,$ $t\ne0,$ analytic.
Denote by $\d_t^*,\db_t^*$ the respective formal $g_t$-adjoints of $\d,\db$
and following Kodaira and Spencer \cite[\S6]{KS}
introduce the smooth family of elliptic operators
\[E_t:=\d\db\db_t^*\d_t^*+\db_t^*\d_t^*\d\db+\db_t^*\d\d_t^*\db+
\d_t^*\db\db_t^*\d+\d_t^*\d+\db_t^*\db.\]
Since $g_0$ is K\"ahler it follows readily that 
$E_0$ is apart from the last two terms just the squared Laplacian.
The last two terms are so added that the kernel of $E_t$ consists of closed forms.
Denote by $\om_t'$ the projection of $\om$ onto the kernel of $E_t$ in the space of $(1,1)$-forms.
Then by Kodaira--Spencer's result \cite[Proposition 8]{KS} the kernel of $E_t$ has dimension independent of $t.$
So by another result of them \cite[Theorem 5]{KS} $\om_t'$ is smooth with respect to $t.$
Putting $\om_t:=\hf(\om_t'+\ov\om_t')$ we see that $\om_t$ is such as we want.

In the case (ii) we write $\om=dJdp$ for some $p:X\to\R$ and perturb $p$ to a real analytic function.
In the case (iii) we take a Stein open set in $(X,J)$ on which $\om$ is exact and which contains the generically immersed Lagrangian, say $L_1.$
This is possible because every self-intersection point of $L_1$ is a transverse double point;
the Stein structure near it is given by taking the product of squared distances from the two components of the generically immersed Lagrangian, which is a plurisubharmonic function.  

In every case there are K\"ahler perturbations $\om_t$ of $\om$ and by Moser's theorem diffeomorphisms $\Ph_t:X\to X$
with $\om=\Ph_t^*\om_t.$
Then $\om$ is analytic with respect to $J_t:=\Ph_t^*J$
and for $i=1,2$ the Lagrangian $\Ph_t^*L_i\sb(X,\om)$ is a nearly-zero Maslov-form
relative to $g_t:=\Ph_t^*g.$
So by Theorem \ref{HamPert1} there is a $g_t$-Maslov-zero Hamiltonian perturbation $L_{it}$ of $\Ph_t^*L_i$
to which we can apply Theorem \ref{Ham} (ii).
Combining it with Lemma \ref{P0} as in the proof of Theorem \ref{P1} we see that Theorem \ref{3.1} holds.  
\end{proof}

Finally using the results of \S\ref{sec3} we get:
\begin{cor}\l{2}
{\bf(i)}
Let $(X,\om,J,g,L_1,L_2)$ be such as in Theorem \ref{P1} or Theorem \ref{3.1} except that $L_1,L_2\sb X$ may be equal.
Let Hypothesis \ref{0} hold and let $\b_1,\b_2\in H\cF(X)$ be two objects supported respectively near $L_1,L_2$ such that:
\e\l{21}
\text{either $\b_1\cong\b_2$ or $k\ge 2n-1$ and $\b_1\cong\b_2$ up to shift.}
\e
Suppose also that 
\e
\l{22}
HF^i(\b_1,\b_1)\cong HF^i(\b_2,\b_2)\ne0\text{ for }i=0,n.
\e
More generally, suppose as in Corollary \ref{7.1} that either: $c_1(X)\le0$ and one of $L_1,L_2$ is embedded;
or $k\ge n+2$ and one of $L_1,L_2$ has on its every self-intersection pair the two gradings nearly equal.
Then $L_1=L_2\sb X.$
\\{\bf(ii)}
Let $(X,\om,J)$ be a K\"ahler manifold equipped with a holomorphic volume form;
and $L_1,L_2\sb X$ two closed irreducibly-immersed special Lagrangians.
Suppose that either $\om$ is $J$-analytic or one of {\rm (i)--(iii)} of Theorem \ref{3.1} holds. 
Let Hypothesis \ref{0} hold and let $H\cF(X)$ have two isomorphic objects supported respectively near $L_1,L_2.$
Then $L_1=L_2\sb X.$
\end{cor}
\begin{proof}
Suppose contrary to the assertion that $L_1\ne L_2.$
Then perturbing $L_1,L_2$ as in Theorem \ref{P1} or Theorem \ref{3.1} 
we see that $HF^*(\b_1,\b_2)$ is supported in degrees $1,\cdots,n-1$ modulo $k\Z.$
But by hypothesis we have $k\ge n$ so $HF^0(\b_1,\b_2)=HF^n(\b_1,\b_2)=0.$ 
If also $\b_1\cong\b_2$ then $HF^0(\b_1,\b_1)=HF^n(\b_1,\b_1)=0$ which however contradicts \eq{22} or Corollary \ref{7.1}. This is the first case of \eq{21}.

In the second case we have $k\ge2n-1$ but $\b_1\cong\b_2$ only up to shift.
Then from Theorem \ref{P1} we see only that
$HF^*(\b_1,\b_1)$ is supported in degrees $i,\cdots,i+n-2$ for some $i\in\Z.$
But again we have $HF^0(\b_1,\b_1)\ne0$ and $HF^n(\b_1,\b_1)\ne0$ so after translating $[i,i+n-2]$ by $k\Z$
we may suppose $0\in[i,i+n-2]$ and $n\in[i+j,i+j+n-2]$ for some $j\in k\Z.$
Then $i\in[2-n,0]\cap[2-j,n-j]$ so $0\ge 2-j$ and $2-n\le n-j;$ that is, $j\in[2, 2n-2].$
So $0<j<2n-1\le k$ which however contradicts $j\in k\Z.$
This proves Corollary \ref{2} (i).

For Corollary \ref{2} (ii) we use Corollary \ref{7} (ii) to see that $L_1,L_2$ have up to shifts the same grading.
We can then apply Corollary \ref{2} (i) with $k=\iy$ which completes the proof.
\end{proof}


\begin{thebibliography}{999}


\bibitem{AJ} M Akaho and D Joyce
`Immersed Lagrangian Floer theory' J.\ Differential Geom.\ 86 (2010) 381--500


\bibitem{Fuk1} K Fukaya
`Cyclic symmetry and adic convergence in Lagrangian Floer theory'
Kyoto J.\ Math.\ 50 (2010) 521--590. 

\bibitem{Fuk2} K Fukaya
`Unobstructed immersed Lagrangian correspondence and filtered $A_\iy$ functor'
arXiv:1706.02131

\bibitem{FOOO} K Fukaya, Y-G Oh, H Ohta and K Ono
`Lagrangian intersection Floer theory---anomaly and obstruction'
Parts I and II AMS/IP Studies in Advanced Mathematics 46.1 and 46.2 (2009)

\bibitem{FOOOt} K Fukaya, Y-G Oh, H Ohta and K Ono
`Lagrangian Floer theory on compact toric manifolds. I' Duke Math.\ J.\ 151 (2010) 23--174


\bibitem{FOOO-Z} K Fukaya, Y-G Oh, H Ohta and K Ono
'Lagrangian Floer theory over integers: spherically positive symplectic manifolds'
Pure Appl.\ Math.\ Q.\
9 (2013) 189--289

\bibitem{HL} R Harvey and H B Lawson Jr
`Calibrated geometries' Acta Math.\ 148 (1982) 47--157

\bibitem{IJO} Y Imagi, D Joyce and J Oliveira dos Santos
`Uniqueness results for special Lagrangians and Lagrangian mean curvature flow in $\C^m$'
Duke Math.\ J 165 (2016) 847--933

\bibitem{Jb} D Joyce
`Riemannian holonomy groups and calibrated geometry'
Oxford Graduate Texts in Mathematics 12 (2007) 

\bibitem{Jc} D Joyce
`Conjectures on Bridgeland stability for Fukaya categories of Calabi--Yau manifolds, special Lagrangians, and Lagrangian mean curvature flow'
EMS Surv.\ Math.\ Sci.\ 2 (2015) 1--62


\bibitem{KS} K Kodaira and D C Spencer
`On deformations of complex analytic structures, III. Stability theorems for complex structures'
Ann.\ of Math.\ 71 (1960) 43--76


\bibitem{Lj} S \L{}ojasiewicz
`Ensembles semi-analytiques'
Inst.\ Hautes \'Etudes Sci.\ 1965



\bibitem{LP}
J D Lotay and T Pacini
`From minimal Lagrangian to J-minimal submanifolds: persistence and uniqueness'
Boll.\ Unione Mat.\ Ital.\ 12 (2019) 63--82
\bibitem{LP2}
J D Lotay and T Pacini
`From Lagrangian to totally real geometry: coupled flows and calibrations' Comm.\ Anal.\ Geom.\ (2020) 607--675

\bibitem{ML} R C McLean
`Deformations of calibrated submanifolds' Comm.\ Anal.\ Geom.\ 6 (1998) 705--747


\bibitem{Sd1} P Seidel
`Graded Lagrangian submanifolds'
Bull.\ Soc.\ Math.\ Fr.\ 128 (2000) 103--149

\bibitem{TY} R P Thomas and S-T Yau
`Special Lagrangians, stable bundles and mean curvature flow'
Comm.\ Anal.\ Geom.\ 10 (2002) 1075--1113

\end{thebibliography}
\end{document}